\documentclass[reqno]{amsart}%
\usepackage{amsmath}
\usepackage{amsfonts}
\usepackage{CJK}
\usepackage{amssymb}
\usepackage{graphicx}%
\usepackage{hyperref}
\usepackage{amscd}
\usepackage{mathrsfs}
\usepackage{bbm}
\usepackage{color}

\newtheorem{thm}{Theorem}[section]

\newtheorem{lem}[thm]{Lemma}

\theoremstyle{definition}
\newtheorem{defn}{Definition}[section]
\theoremstyle{Conjecture}

\theoremstyle{remark}
\newtheorem{rem}{Remark}[section]
\theoremstyle{Example}

\newcommand{\be}{\begin{equation}}
\newcommand{\ee}{\end{equation}}
\newcommand{\bea}{\begin{eqnarray}}
\newcommand{\eea}{\end{eqnarray}}
\newcommand{\ben}{\begin{eqnarray*}}
\newcommand{\een}{\end{eqnarray*}}
\newcommand{\bet}{\begin{equation}
\begin{split}}
\newcommand{\eet}{\end{split}
\end{equation}}

\begin{document}
\title[A Characterization of regular points by $L^2$ Extension Theorem]
{A Characterization of regular points by\\ $L^2$ Extension Theorem}

\author{Qi'an Guan}
\address{Qi'an Guan: School of Mathematical Sciences, and Beijing International Center for Mathematical Research,
Peking University, Beijing, 100871, China.}
\email{guanqian@amss.ac.cn}
\author{Zhenqian Li}
\address{Zhenqian Li: School of Mathematical Sciences, Peking University, Beijing, 100871, China.}
\email{lizhenqian@amss.ac.cn}

\thanks{The first author was partially supported by NSFC-11522101 and NSFC-11431013.}

\date{\today}

\begin{abstract}
In this article, we present that the germ of a complex analytic set at the origin in $\mathbb{C}^n$ is regular if and only if the related $L^2$ extension theorem holds. We also obtain a necessary condition of the $L^2$ extension of bounded holomorphic sections from singular analytic sets.
\end{abstract}
\maketitle

\section{introduction}\label{sec:introduction}
Let $M$ be a Stein manifold and $X\subset M$ a closed complex subspace. Cartan extension theorem says that any holomorphic function $f$ on $X$ can be extended to a holomorphic function $F$ on the Stein manifold $M$. Then, it is natural to ask that if the holomorphic function $f$ has some special property, whether we can find an extension $F$ possessing the same property. In \cite{O-T}, Ohsawa and Takegoshi considered the extension of $L^2$ holomorphic functions. More precisely, they proved the following famous Ohsawa-Takegoshi $L^2$ extension theorem:

\begin{thm} \emph{(\cite{O-T}).} \label{O-T}
Let $\Omega$ be a bounded pseudoconvex domain in $\mathbb{C}^n$. Let $\varphi$ be a plurisubharmonic function on $\Omega$. Let $H$ be an $m$-dimensional complex plane in $\mathbb{C}^n$. Then for any holomorphic function on $H{\cap}\Omega$ satisfying $$\int_{H{\cap}\Omega}|f|^2e^{-2\varphi}d\lambda_H<\infty,$$ there exists a holomorphic function $F$ on $\Omega$ such that $F|_{H{\cap}\Omega}=f$ and
$$\int_\Omega|F|^2e^{-2\varphi}d\lambda_n\leq C_\Omega\cdot\int_{H{\cap}\Omega}|f|^2e^{-2\varphi}d\lambda_H,$$
where $d\lambda_H$ is the Lebesgue measure, and $C_\Omega$ is a constant which only depends on the diameter of $\Omega$ and $m$.
\end{thm}

It is natural to ask:

\noindent{\textbf{Question.}} Let $\Omega\subset\mathbb{C}^n$ be a domain and $A\subset\Omega$ an analytic set through the origin $o$. If the above $L^2$ extension theorem holds for any bounded pseudoconvex domain $\tilde\Omega\ni o$ such that $A\cap\tilde\Omega$ is an analytic set in $\tilde\Omega$, can one obtain that $o$ is a regular point of $A$?

In this article, we will present a positive answer, i.e.,

\begin{thm} \label{main1}
Let $\Omega\subset\mathbb{C}^n\ (n\geq2)$ be a domain, $A\subset\Omega$ an analytic set through the origin $o$. Then, for small enough ball $B_{r}(0)\subset\Omega$, the $L^2$ extension theorem holds for $(B_{r}(0), A)$ if and only if $o$ is a regular point of $A$.
\end{thm}

We also present a necessary condition of the $L^2$ extension of bounded holomorphic sections from singular analytic sets as follows:

\begin{thm} \label{main2}
Let $\Omega\subset\mathbb{C}^n\ (n\geq2)$ be a domain and $o\in\Omega$ the origin. Let $A\subset\Omega$ be an analytic set through $o$ with $\dim_oA=d\ (1\leq d\leq n-1)$. If the germ $(A,o)$ of $A$ at $o$ is reducible or \emph{ord}$\mathscr{I}_{A,o}:=\min\{\mbox{\emph{ord}}_o(f)|f\in\mathscr{I}_{A,o}\}\geq d+1$, then there exists a small enough ball $B_{r_0}(0)\subset\Omega$, a holomorphic function $f$ on $B_{r_0}(0){\cap}A$ and a plurisubharmonic function $\varphi$ on $B_{r_0}(0)$ with bounded $|f|^2e^{-2\varphi}$ such that, for any $r<r_0$, $f$ has no holomorphic extension $F$ to $B_r(0)$ satisfying $$\int_{B_r(0)}|F|^2e^{-2\varphi}d\lambda_n<\infty,$$
\end{thm}

In particular, we can take $A$ to be hypersurfaces with Brieskorn singularities in $\mathbb{C}^n$, i.e., $A:=\{z_1^{\alpha_1}+z_2^{\alpha_2}+\cdots+z_m^{\alpha_m}=0\}\subset{\mathbb{C}^n}$, where $2\leq m\leq n,\ \alpha_k\geq n$ are positive integers.

\begin{rem}
In \cite{D-M}, Diederich and Mazzilli gave an example with a surface $A$ defined by equation $z_1^2+z_2^q=0$ in $\mathbb{C}^3$, where $q>3$ is any fixed uneven integer. Moreover, Ohsawa also presented an example with $A=\{z_1z_2=0\}\in\mathbb{C}^2$ in \cite{O}.
\end{rem}

\section{proof of main results}

For the convenience, firstly we recall the following notion of integral closure of ideals.

\begin{defn} (see \cite{LJ-T}).
Let $R$ be a commutative ring and let $I$ be an ideal of $R$. An element $h\in R$ is said to be integrally dependent on $I$ if it satisfies a relation
\[ h^d+a_1h^{d-1}+...+a_d=0 \quad (a_i\in{I}^i, 1\le i\le d). \]

The set $\bar{I}$ consisting of all elements in $R$ which are integrally dependent on $I$ is called the integral closure of $I$ in $R$, which is an ideal of $R$. $I$ is called integrally closed if $I=\bar{I}$.
\end{defn}

To prove main results, we need the following Skoda's division theorem.

\begin{thm} \emph{(see \cite{De}, Chapter VIII, Theorem 9.10).} \label{Skoda}
Let $\Omega$ be a pseudoconvex open subset of $\mathbb{C}^n$, let $\varphi$ be a plurisubharmonic function and $g=(g_1,...,g_r)$ be a $r$-tuple of holomorphic functions on $\Omega$. Set $m=\min\{n,r-1\}$. Then for every holomorphic function $f$ on $\Omega$ such that
$$I=\int_{\Omega}|f|^{2}|g|^{-2(m+1+\varepsilon)}e^{-\varphi}d\lambda_n<\infty,$$
there exist holomorphic functions $(h_1,...,h_r)$ on $\Omega$ such that $f=\sum\limits_{k=1}^{r}h_kg_k$ and
$$\int_{\Omega}|h|^{2}|g|^{-2(m+\varepsilon)}e^{-\varphi}d\lambda_n\leq(1+m/\varepsilon)I,$$
where $|g|^2=|g_1|^2+|g_2|^2+\cdots+|g_r|^2$.
\end{thm}

Moreover, the following strong openness property of multiplier ideal sheaves is also necessary.

\begin{thm} \emph{(\cite{G-Z_open_a, G-Z_open}).} \label{SOC}
Let $\varphi$ be a plurisubharmonic function on complex manifold $X$ and $\mathscr{I}_+(\varphi):=\cup_{\varepsilon>0}\mathscr{I}((1+\varepsilon)\varphi)$. Then
$$\mathscr{I}_+(\varphi)=\mathscr{I}(\varphi),$$
where $\mathscr{I}(\varphi)$ is the sheaf of germs of holomorphic functions $f$ such that $|f|^2e^{-\varphi}$ is locally integrable.
\end{thm}

\begin{lem} \label{finiteness}
Let $\Omega\subset\mathbb{C}^n\ (n\geq2)$ be a domain and $A\subset\Omega$ an analytic set with pure dimension $d$ through the origin $o$. Then, there exists a neighborhood $U$ of $o$ such that
$$\int_{U{\cap}A}(|z_1|^2+\cdots+|z_n|^2)^{-(d-1)}dV_A<\infty$$
where $dV_A=(\omega^d|_{A_{reg}})/d!,\ \omega=\frac{\sqrt{-1}}{2}\sum\limits_{k=1}^{n}dz_k{\wedge}d\bar{z}_k$.
\end{lem}

\begin{proof}
Note that the form $\frac{\omega^d}{d!}$ can be written as $\frac{\omega^d}{d!}={\sum\limits_{\#I=d}}\mspace{-8mu}'dV_I$, where $I=(k_1,...,k_d)$, $dV_I$ denotes the volume form $\prod\limits_{\alpha=1}^{d}\frac{\sqrt{-1}}{2}dz_{k_{\alpha}}{\wedge}d\bar{z}_{k_{\alpha}}$ in the coordinate plane $\mathbb{C}_I$ and $\sum\limits_{\#I=d}\mspace{-8mu}'$ represents the summation over the ordered multi-indices of length $d$. Let $w=Tz$ be a unitary transformation of coordinates satisfying, in the coordinates $w=(w_1,...,w_n)$, there is a bounded neighborhood $U_I$ of $o$ such that the projection $\pi_I:U_I{\cap}A\to U_I'=U_I{\cap}\mathbb{C}_I$ is a branched covering with the number of sheets $s_I$ for every $I$ with $\#I=d$ (see \cite{Chirka}, p.33, Lemma 2). Thus, we have
\begin{equation*}
\begin{split}
&\int_{T(U_I){\cap}A_{reg}}|z|^{-2(d-1)}dV_{z,I}=\int_{U_I{\cap}T^{-1}(A_{reg})}|w|^{-2(d-1)}dV_{w,I}\\
&=s_I\int_{U_I'}|w|^{-2(d-1)}dV_{w,I}\leq s_I\int_{U_I'}(|w_{k_1}|^2+\cdots+|w_{k_d}|^2)^{-(d-1)}dV_{w,I}<\infty.
\end{split}
\end{equation*}

Let $U=T(\cap U_I)$. Then, we obtain
$$\int_{U{\cap}A}(|z_1|^2+\cdots+|z_n|^2)^{-(d-1)}dV_A
\leq{\sum\limits_{\#I=d}}\mspace{-8mu}'\int_{U_I{\cap}T^{-1}(A_{reg})}|w|^{-2(d-1)}dV_{w,I}<\infty.$$
\end{proof}

We are now in a position to prove our main results.\\

\noindent{\textbf{Proof of Theorem \ref{main1}.} It is enough to prove the necessity.

Without loss of generality, we can assume $1\leq\dim_oA=d\leq n-1$, and $(A,o)$ is irreducible by Remark \ref{reducible}.

Suppose that $o$ is a singular point of $A$. It follows from the local parametrization theorem of analytic sets that there is a local coordinate system $(z';z'')=(z_1,...,z_d;z_{d+1},...,z_n)$ near $o$ such that for some constant $C>0$, we have $|z''|\leq C|z'|$ for any $z\in A$ near $o$.

Let $\mathcal{I}\subset\mathcal{O}_{A,o}$ be the ideal generated by germs of holomorphic functions $\bar{z}_1,...,\bar{z}_{d}\in\mathcal{O}_{A,o}$, where $\mathcal{O}_A=\mathcal{O}_{\Omega}/\mathscr{I}_A\big|_A$ and $\bar{z}_k$ are the residue classes of $z_k$ in $\mathcal{O}_{A,o}$. Since $o$ is a singularity of $A$, the embedding dimension $\dim_{\mathbb{C}}\mathfrak{m}_{A,o}/\mathfrak{m}^2_{A,o}$ of $A$ at $o$ is at least $d+1$ (see \cite{De}, Chapter II, Proposition 4.32), which implies that there exists $d+1\leq k_0\leq n$ such that $\bar{z}_{k_0}\not\in\mathcal{I}$.

It follows from $|z''|\leq C|z'|$ for any $z\in A$ near $o$ that $|z_{k_0}|^2\leq C^2|z'|^2$ and $|z'|^2\geq\frac{1}{1+C^2}|z|^2$ on $U\cap A$ for some neighborhood $U$ of $o$. By Lemma \ref{finiteness}, for some smaller neighborhood $U$ of $o$, we have
$$\int_{U{\cap}A}|{z}_{k_0}|^2|z'|^{-2d}dV_A\leq C^2(1+C^2)^{d-1}\int_{U{\cap}A}|z|^{-2(d-1)}dV_A<\infty.$$
Take a small ball $B_r(0)\subset U$. It follows from the $L^2$ extension theorem that there exists a holomorphic function $F\in\mathcal{O}(B_r(0))$ such that $F|_A=\bar{z}_{k_0}$ and $$\int_{B_r(0)}|F|^2|z'|^{-2d}d\lambda_n<\infty.$$
By Theorem \ref{SOC}, for sufficiently small $\varepsilon>0$ and smaller $B_r(0)$ we have
$$\int_{B_r(0)}|F|^2|z'|^{-2(d+\varepsilon)}d\lambda_n<\infty.$$
Then, we infer from Theorem \ref{Skoda} that there exist holomorphic functions $f_k\in\mathcal{O}(B_r(0))$ such that $F=\sum_{k=1}^{d}f_k{\cdot}z_k$, i.e., $F\in(z_1,...,z_{d})\cdot\mathcal{O}_n$.
By restricting to $A$, we have $\bar{z}_{k_0}\in\mathcal{I}$, which contradicts to $\bar{z}_{k_0}\not\in\mathcal{I}$.
\hfill $\Box$

\begin{rem} \label{reducible}
Ohsawa's argument in \cite{O} implies that if $(A,o)$ is reducible, then, for any small ball $B_{r}(0)\subset\Omega$, the $L^2$ extension theorem does not hold for $(B_{r}(0), A)$. In fact, if $(A,o)=(A_1,o)\cup(A_2,o)$ with $(A_i,o)$ are irreducible. Take $f_i\in\mathcal{O}_n$ such that $f_i|_{A_i}\equiv0$ and $f_i|_{A_j}\not\equiv0,\ i\not=j$. Let $\varphi=\log|f_1-f_2|$ and $f=f_1(f_1-f_2)/(f_1+f_2)$. Then, $f|_A=f_1|_A$ and $|f|^2e^{-2\varphi}$ is bounded on $A$ near $o$. The holding of $L^2$ extension theorem implies that there exists a holomorphic function $F\in\mathcal{O}_n$ such that $F=g(f_1-f_2)$ for some $g\in\mathcal{O}_n$ and $F|_A=f$, which implies $gf_2|_{A_1}\equiv0, gf_1|_{A_2}=f_1$. Then, we have $g|_{A_1}\equiv0$ and $g|_{A_2}\equiv1$, which is impossible.
\end{rem}

\noindent{\textbf{Proof of Theorem \ref{main2}.} By Remark \ref{reducible}, it is sufficient to prove the case that $(A,o)$ is irreducible and ord$\mathscr{I}_{A,o}\geq d+1$. It follows from $\dim_oA=d$ and Proposition 4.8 of Chapter II in \cite{De} that, in some local coordinates $(z';z'')=(z_1,...,z_d;z_{d+1},...,z_n)$ near $o$, there exist Weierstrass polynomials
\\
\hspace*{\fill} $P_k=z_k^{m_k}+a_{1k}z_k^{m_k-1}+\cdots+a_{m_kk}\in\mathcal{O}_{k-1}[z_k]\cap\mathscr{I}_{A,o},\ k=d+1,...,n$. \hspace*{\fill} $(*)$
\\
with $m_k=\mbox{ord}_oP_k$. Hence, we have
\\
\hspace*{\fill} $a_{jk}(z_1,...,z_{k-1})\in\mathfrak{m}_{k-1}^j,\ d+1\leq k\leq n,\ 1\leq j\leq m_k$. \hspace*{\fill} $(**)$

Consider the ideal $\mathcal{I}$ in $\mathcal{O}_{A,o}$ generated by germs of holomorphic functions $\bar{z}_1,...,\bar{z}_\lambda\in\mathcal{O}_{A,o}$, where $\bar{z}_k$ are the residue classes of $z_k$ in $\mathcal{O}_{A,o}$ and $d\leq\lambda\leq\min\{\mbox{ord}\mathscr{I}_{A,o}-1,n-1\}$. Then, combining $(*)$ and $(**)$, we obtain that the integral closure $\overline{\mathcal{I}}$ of $\mathcal{I}$ in $\mathcal{O}_{A,o}$ is $\mathfrak{m}_{A,o}=(\bar{z}_1,...,\bar{z}_{n})\cdot\mathcal{O}_{A,o}$, the maximal ideal of $\mathcal{O}_{A,o}$. Moreover, since ord$\mathscr{I}_{A,o}\geq\lambda+1$, we have $(\bar{z}_k)^\lambda\not\in\mathcal{I},\ \lambda+1\leq k\leq n$. In particular, $(\overline{\mathcal{I}})^{\lambda}\not\subset\mathcal{I}$.

Let $B_r(0)\subset\Omega$ be a small enough ball such that all $P_k,\bar{z}_k$ are holomorphic on $A{\cap}B_r(0)$. Let $\hat{g}_k, 1\leq k\leq\lambda$, be arbitrarily holomorphic extension of $g_k$ to $B_r(0)$ with $g_k=\bar{z}_k$ and $\varphi=\frac{\lambda}{2}\log|\hat{g}|^2$. Since $\mathcal{O}_{A,o}$ is reduced, for any $(f,o)\in(\overline{\mathcal{I}})^\lambda$, we have $|f|\leq C\cdot|g|^\lambda$ for some constant $C>0$ by Theorem 2.1 vi) in \cite{LJ-T}. Hence, for some small ball $B_{r_0}(0)$, we can assume that on $A{\cap}B_{r_0}(0)$, $f$ is holomorphic and $|f|^2\cdot e^{-2\varphi}$ is bounded.

Suppose that we have a $L^2$ extension $F\in\mathcal{O}(B_r(0))$ with some $r<r_0$ such that $F|_A=f$ and $$\int_{B_r(0)}|F|^2|\hat{g}|^{-2\lambda}d\lambda_n<\infty.$$
It follows from Theorem \ref{SOC} that for sufficiently small $\varepsilon>0$ and smaller $B_r(0)$ we have
$$\int_{B_r(0)}|F|^2|\hat{g}|^{-2(\lambda+\varepsilon)}d\lambda_n<\infty.$$
By Theorem \ref{Skoda}, there exist holomorphic functions $f_k\in\mathcal{O}(B_r(0))$ such that $F=\sum_{k=1}^{\lambda}f_k\cdot\hat{g}_k$, which implies $(F,o)\in(\hat{g}_1,...,\hat{g}_\lambda)\cdot\mathcal{O}_n$.
By restricting to $A$, we have $(f,o)\in\mathcal{I}$. As $(f,o)$ is arbitrary, we obtain $(\overline{\mathcal{I}})^\lambda\subset\mathcal{I}$, which contradicts to $(\overline{\mathcal{I}})^\lambda\not\subset\mathcal{I}$.

\hfill $\Box$

\vspace{.1in} {\em Acknowledgements}. The authors would like to sincerely thank our supervisor, Professor Xiangyu Zhou, and his seminar for bringing us to the $L^2$ extension problem in several complex variables and for his valuable help to us in all way.

The authors would also like to sincerely thank Professor Takeo Ohsawa for giving talks on related topics at CAS and sharing his works.




\begin{thebibliography}{123}
\bibitem{Chirka} E. M. Chirka, \emph{Complex analytic sets}, Translated from the Russian by R. A. M. Hoksbergen, Mathematics and its Applications (Soviet Series), 46, Kluwer Academic Publishers Group, Dordrecht, 1989.
\bibitem{De} J.-P. Demailly, \emph{Complex Analytic and Differential Geometry}, electronically accessible at
    http://www-fourier.ujf-grenoble.fr/$\sim$demailly/documents.html. Institut Fourier (2012).
\bibitem{D-M} K. Diederich, E. Mazzilli, \emph{A remark on the theorem of Ohsawa-Takegoshi}, Nagoya Math. J. 158 (2000), 185--189.
\bibitem{G-Z_open_a} Q. A. Guan, X. Y. Zhou, \emph{Strong openness conjecture for plurisubharmonic functions}, arXiv:1311.3781.
\bibitem{G-Z_open}  Q. A. Guan, X. Y. Zhou, \emph{A proof of Demailly's strong openness conjecture}, Ann. of Math. 182 (2015), 605--616.
\bibitem{LJ-T} M. Lejeune-Jalabert, B. Teissier, \emph{Cl\^{o}ture int\'{e}grale des id\'{e}aux et \'{e}quisingularit\'{e}}, Ann. Fac. Sci. Toulouse Math., Vol. 17, No. 4 (2008), 781--859.
\bibitem{O} T. Ohsawa, \emph{On a curvature condition that implies a cohomology injectivity theorem of Koll\'{a}r-Skoda type}, Publ. Res. Inst. Math. Sci. 41 (2005), no. 3, 565--577.
\bibitem{O-T} T. Ohsawa, K. Takegoshi, \emph{On the extension of $L^2$ holomorphic functions}, Math. Z. 195 (1987), 197--204.
\end{thebibliography}
\end{document}